\documentclass[a4paper,12pt,draft]{article}

\usepackage[top=1in,left=1in,right=1in,bottom=1.5in]{geometry}

\usepackage[T1]{fontenc}
\usepackage{amsmath,amssymb,amsfonts,amsthm,mathrsfs,eucal,dsfont,verbatim}

\usepackage{color}

\newcommand{\normal}{\color{black}}

\theoremstyle{plain}
\newtheorem{theorem}{Theorem}

\newtheorem{corollary}{Corollary}
\theoremstyle{definition}
\newtheorem{remark}{Remark}

\newtheorem{example}{Example}

\numberwithin{equation}{section}

\newcommand{\nat}{\mathds{N}}
\newcommand{\real}{\mathds{R}}

\newcommand{\I}{\mathds{1}}

\newcommand{\RR}{\mathrm{Re}\,}
\newcommand{\II}{\mathrm{Im}\,}

\title{On the speed of convergence in the local limit theorem for triangular arrays of random variables}

\author{%
    \textsc{V. Knopova}%
    \thanks{V.M.\ Glushkov Institute of Cybernetics,
            NAS of Ukraine,
            40, Acad.\ Glushkov Ave.,
            03187, Kiev, Ukraine,
            \texttt{vic\underline{ }knopova@gmx.de}}
    }

\date{}

\begin{document}

\maketitle
\begin{abstract}
    \noindent

    We establish the upper bound on the speed of convergence to the infinitely divisible limit density in the local limit theorem for triangular arrays of random variables $\{X_{k,n},\, k=1,..,a_n, \, n\in \nat\}$.


\end{abstract}

\maketitle

\section{Introduction}

This paper is motivated by \cite{KKu07}, where the local limit theorem for a triangular array of independent and identically distributed (i.i.d) in each series random variables $\{X_{k,n},\, k=1,..,a_n\}$ is established. Staying in frames of the situation studied in \cite{KKu07}, in this note  we would like make a step further and  obtain the information about the speed of convergence to the limit density.

In contrast to the local limit theorem for the normal law, there is not much known even about the local limit theorem for infinitely divisible limit densities. Of course, one can refer to Gnedenko's theorem on the necessary and sufficient conditions for the convergence to the stable law, see \cite{IL71}. Under certain conditions the uniform convergence to the limit density was proved in \cite{KKu07},   but up to the author's knowledge  in the general case nothing is known  about the speed of convergence.

To make the presentation self-contained, we quote  below the necessary 
and sufficient conditions for convergence to the infinitely divisible law, see  Theorem~2, Chapter XVII $\S2$ from 
\cite{Fe71}. 

Recall that a measure $M$ on $\real$ is called {\itshape canonical} if $M(I)<\infty$ for any finite interval, and
$$
M^+(x)=\int_x^{+\infty} \frac{1}{u^2} M(du)<+\infty, \quad M^-(x) =\int_{-\infty}^{-x} \frac{1}{u^2} M(dy)<+\infty, \quad x>0.
$$
A sequence of canonical measures $\{M_n\}$ converges to a canonical measure \emph{properly}  if $M_n(I)\to M(I)$ for any finite interval, and $M_n^+(x)\to M^+(x)$, $M_n^{-}(x)\to M^{-}(x)$ for every $x>0$. In this case we write $M_n\to M$.

\begin{theorem}\cite{Fe71}\label{fel1}
Let $\{X_{k,n}\, 1\leq k\leq a_n\}$ be such that $X_{k,n}$ are  i.i.d. for any $1\leq k\leq a_n$, $a_n\to\infty$ as $n\to\infty$, and satisfy 
  \begin{equation}
\lim_{n\to\infty}  P\{|X_{1,n}|\geq \varepsilon\}=0.\label{small}
\end{equation}
Let $F_n(du)$ be the distribution function of $X_{1,n}$, 
$$
M_n(du):=a_n u^2 F_n(du),\quad \beta_n:= \int_\real \sin u F_n(du).
$$
 Then $S_n:= X_{1,n}+...+X_{a_n,n}$ converges in distribution to some random variable $S$ if and only of
\begin{equation}
M_n\to M, \quad a_n\beta_n\to \beta, \quad \text{as} \quad n\to\infty \label{conv-mes}
\end{equation}
for some $\beta\in \real$ and some canonical measure $M$. In this case the characteristic function $\Phi(z)$ of $S$ is given by
\begin{equation}
\Phi(z)=\exp\left\{ i\beta z +\int_\real \frac{e^{izu}-1-iz\sin u}{u^2} M(du)\right\}. \label{LH1}
\end{equation}
\end{theorem}

The function
\begin{equation}
\psi(z):= -i\beta z +\int_\real \frac{1-e^{izu}+iz\sin u}{u^2} M(du)=-i\beta z + \phi(z)\label{LH2}
 \end{equation}
is called the \emph{characteristic exponent} of the infinitely divisible variable $S$. Put
\begin{equation}
\psi_n(z):= -i\beta_n  z +\int_\real \frac{1-e^{izu}+iz\sin u}{u^2} M_n(du)= -i\beta_n  z +\phi_n(z).  \label{LH3}
\end{equation}

\begin{remark}\label{rem1}
 Of course, one can formulate Theorem~\ref{fel1} with the function $\I_{|u|\leq1}$ instead of $\sin u$ under the integral, but for technical convenience we need the L\'evy representation \eqref{LH1}.
\end{remark}
Sometimes, especially when the convergence in Theorem~\ref{fel1} is that to a stable law (cf. \cite{IL71}) it is more convenient to consider the random variables in the form 
\begin{equation}
X_{k,n}= \frac{\xi_{k,n}}{b_n}, \label{X}
\end{equation}
where  the variables  $\xi_{k,n}$, $1\leq k\leq a_n$,  are i.i.d. for each $n$, and the sequence $(b_n)_{n\geq 1}$ satisfies certain growth assumptions.
In the paper we assume that the random variables $X_{k,n}$ are of the form \eqref{X}. We 
 assume that the conditions of Theorem~\ref{fel1} hold true, and  thus 
$S_n= \tfrac{\xi_{1,n}+..+\xi_{k,n}}{b_n}$ converges weakly as $n\to \infty$ to some infinitely divisible distribution $S$. Under some conditions on the sequences $(a_n)_{n\geq1}$, $(b_n)_{n\geq1}$,   and on the distribution of $\xi_{1,n}$ (cf. \cite{KKu07}), $S_n$ and $S$ possess transition probability   densities and the local limit theorem takes place. Taking this result as the starting point we derive  in Theorem~\ref{lt-main}  (under some additional assumptions) the speed of convergence to the limit density, and illustrated our result by examples.

In order to make the presentation as transparent as possible, we write the  main notation in  Table 1. Finally,  denote by $\hat{m}$ the symmetrization of the measure $m$, i.e.
$\hat{m}(A):= \frac{m(A)+m(-A)}{2}$ for any Borel set $A\in \real$.

\begin{table}\label{1}
    \caption{Notation}\centering
    \begin{tabular}{|c| c |c |c|}
    \hline \hline
    Variable & Char. function & Probab. measure & Probab. density\\ [0.5ex]
    \hline
    $\xi_{1,n}$ & $\theta_n(z)$ & $G_n(du)$ & $g_n(u)$ \\
    $X_{1,n}$ & $\theta_n(z/b_n)$ & $F_n(du)\equiv G_n(b_n du)$ & $f_n(u)\equiv  b_n g_n(b_n u)$  \\
    $S_n$ & $\Phi_n(z)\equiv \theta_n^{a_n}(z/b_n)$ & $P_n(dx)$& $p_n(x)$\\
    $S$ & $\Phi(z)=e^{-\psi(z)}$ & $P(dx)$ & $p(x)$\\
    \hline
    \end{tabular}
\end{table}


\section{Main result}

We assume that the assumptions below hold true:

\textbf{A.}  for any $n\geq 1$ the variable $\xi_{1,n}$  possesses the density $g_n(x)$;

\textbf{B.}  $\exists \alpha\in (0,2)$ such that $\RR\psi(z)\geq c |z|^\alpha$ for $|z|$ large enough;

\textbf{C.}   $\forall\delta>0$ we have  $N(\delta):= \underset{n\geq 1,\, |z|\geq \delta}{\sup} |\theta_n(z)|<1$;

\textbf{D.} $\sup_{n\geq 1} \int_\real g_n^2(x)dx<\infty$;

\textbf{E.} $b_n\to \infty$, $\frac{\ln b_n}{a_n} \to 0$ as  $n\to\infty$;

\textbf{F}. for $n\geq 1$ one of the conditions below is satisfied:

a) there exists $c(\delta)>0$, $0<\kappa<2$, such that 
\begin{equation}
 a_n \int_{\real} (1-\cos(zu)) \hat{F}_n(du) \geq  c(\delta)|z|^\kappa \quad \text{for all $|z|\leq \delta b_n$;} \label{3a}
 \end{equation}

b) $\hat{F}\leq \hat{F}_n$ on $\real$.

\textbf{G}.   $\exists \delta>0$ such that  $\underset{n\geq 1,\, |z|\leq \delta}{\inf}  |\mathrm{Re} \theta_n(z)|>0$.

\medskip
\begin{remark}\label{rem2}
Condition \textbf{A}, and \textbf{C}--\textbf{E} are taken from  \cite{KKu07}. Instead of condition \textbf{B} in \cite{KKu07} another condition was assumed (namely, a version of the  Kallenberg condition  \cite{Ka81} for the sequence of measures $\hat{M}_n$), which in fact implies \textbf{B}.
\end{remark}

Let
 \begin{equation}
 \begin{split}
 \gamma_n':&=  \sup_{z\in \real} \frac{|\RR \phi(z)-\RR \phi_n(z)|}{1+z^2},  \quad \gamma_n'':=  \sup_{z\in \real} \frac{|\II \phi(z)- \II \phi_n(z)|}{1+z^2},
 \\
  \chi_n:&= |a_n\beta_n-\beta|,
  \end{split}\label{gg}
 \end{equation}
 where $\beta_n$ has the same meaning as in Theorem~\ref{fel1}. 
 
From now  we fix $\delta>0$, for which the above conditions  hold true.
 For  some fixed $0<\epsilon<1$   put
 \begin{equation}
 \rho_{\epsilon, \delta}(n):=\max \left(\chi_n, \gamma_n',\gamma_n'', a_n^{-1}, e^{a_n(\ln N(\delta)+\epsilon) } e^{-(1-\epsilon)\RR \psi(\delta b_n)}\right),\label{rho1}
 \end{equation}
where $N(\delta)$ is defined in \textbf{C}.

\begin{theorem}\label{lt-main}
Suppose that conditions \eqref{small}, \eqref{conv-mes}, and \textbf{A}--\textbf{G}  are satisfied.  Then the distributions $S_n$ and $S$ possess, respectively, the densities $p_n(x)$ and $p(x)$, and
\begin{equation}
\sup_{x\in \real} |p_n(x)-p(x)| \leq C  \rho_{\epsilon, \delta}(n), \quad n\to\infty, \label{lt-eq}
\end{equation}
where $\rho_{\epsilon, \delta}(n)$ is given by \eqref{rho1}.
\end{theorem}

One can simplify the expression for the speed of convergence, but at the expense of some additional assumptions on $a_n$ and $b_n$. We say that a sequence $(c_n)_{n\geq 1}$ satisfies condition 
\textbf{H},  if there exist $a,b>0$ such that
$$
0<\liminf_{n\to \infty} \frac{c_n}{n^a}\leq \limsup_{n\to\infty} \frac{c_n}{n^b}<\infty. 
$$

\begin{corollary}\label{cor2}
Suppose conditions of Theorem~\ref{lt-main} hold true, and assume in addition that the  sequences $(a_n)_{n\geq 1}$ and $(b_n)_{n\geq 1}$  satisfy \textbf{H}.  Then
\begin{equation}
\sup_{x\in \real} |p_n(x)-p(x)|\leq C \max\Big(\gamma_n',\gamma_n'',\chi_n,a_n^{-1}\Big). \label{eq-rho3}
\end{equation}
\end{corollary}

\begin{corollary}\label{cor3}
Suppose that conditions \textbf{A}--\textbf{F}  and \textbf{H}  hold true, the densities  $p_n(x)$ and $p(x)$ are symmetric, and 
\begin{equation}
\Phi_n(z)\geq \Phi(z) \quad \forall n\geq 1, \label{P}
\end{equation}
uniformly in $\{z: \,\, |z|\leq \delta b_n\}$. Then 
\begin{equation}
\sup_{x\in \real} |p_n(x)-p(x)|\leq C \Big(\gamma_n'+ r(n))\Big), \label{eq-rho4}
\end{equation}
where $r(n)=o(n^{-k}) $ as $n\to \infty$ for any $k\geq 1$. 
\end{corollary}

\begin{remark}
 As one can expect, the oscillation of measures involved in $\gamma_n'$ and $\gamma_n''$ can play the crucial role in the estimation of the speed of convergence. For example, it might be insufficient to know the behaviour of such a  "rough estimate" for  $\gamma_n'$ as below:
$$
\sup_{z\in \real } \frac{\Big|\int_\real (1\wedge |uz|^2) (M_n(du)-M(du))\Big|}{1+z^2},
$$
in particular, when the densities $(g_n)_{n\geq 1}$ have oscillations. Such a situation is illustrated in  Example~\ref{ex1}.


\end{remark}

\section{Proofs}

\begin{proof}[Proof of Theorem~\ref{lt-main}]

Recall that the densities $p_n(x)$ and $p(x)$ can be written as the inverse Fourier transforms of the respective characteristic functions:
\begin{equation}
p(x)=(2\pi)^{-1}\int_\real e^{-iz x}\Phi(z)dz=(2\pi)^{-1}\int_\real e^{-iz x-\psi(z)}dz,\label{p}
\end{equation}
\begin{equation}
p_n(x)=(2\pi)^{-1} \int_\real e^{-izx} \Phi_n(z) dz.\label{pn}
\end{equation}
By \eqref{p} and \eqref{pn} we have 
\begin{align*}
\Delta_n :&= 2\pi \sup_{x\in \real}|p_n(x)-p(x)|\leq  \int_\real |\Phi_n(z)-\Phi(z)|dz
\\&
\leq  \Big(\int_{-\delta b_n }^{\delta b_n} |\Phi_n(z)-\Phi(z)|dz+ \int_{|z|>\delta b_n} |\Phi_n(z)|dz+\int_{|z|>\delta b_n} |\Phi(z)|dz\Big)
\\&
=: I_1(n)+I_2(n)+I_3(n),
\end{align*}
where $\delta>0$. We estimate the terms $I_k(n)$, $k=1,2,3$, separately.

{\itshape Estimation of  $I_1$}. Observe, that
\begin{equation}
\begin{split}
|1-e^{x+iy}|&= |1+e^{2x}-2e^x\cos y|^{1/2}
\\&
= |(1-e^x)^2 +2e^x (1-\cos y)|^{1/2}  
\\&
\leq |(1-e^x)^2 +e^x y^2|^{1/2}
\\&
\leq |1-e^x|+e^{x/2} |y|
\\&
\leq e^{x_+} (|x|+|y|),\label{es1}
\end{split}
\end{equation}
where $x,y\in \real$, and $x_+:= \max (x,0)$.
Denote
\begin{equation}
H_n(z):= \psi(z)+ a_n \ln \theta_n\big(\frac{z}{b_n}\big).\label{Hnz}
\end{equation}
Then by \eqref{es1} we get
$$
I_1(n)\leq \int_{-\delta b_n}^{\delta b_n} e^{-\RR \psi(z)+ (\RR H_n(z))_+} (|\RR H_n(z)|+ |\II H_n(z)|)dz.
$$
Since $\ln (1-z)\leq -z$ for $z\in (0,1)$, then
\begin{equation}
\mathrm{Re}H_n(z)= \mathrm{Re}\,\psi(z)+ a_n \ln \big|\theta_n\big(\frac{z}{b_n}\big)\big|
\leq \mathrm{Re}\,\psi(z) - a_n\big(1-|\theta_n\big(\frac{z}{b_n}\big)\big|\big).\label{reh1}
\end{equation}
Observe, that
\begin{equation}
\Big|\theta_n\big(\frac{z}{b_n}\big)\Big|=\int_\real \cos (zu) \hat{F}_n(du)=a_n^{-1} \int_\real \frac{\cos (zu)}{u^2} \hat{M}_n(du). \label{thetan}
\end{equation}
Therefore, by \eqref{reh1} and  \textbf{F} we have  for  all $n$ large enough and $|z|\leq \delta b_n$
\begin{equation}
\begin{split}
\mathrm{Re}\,\psi(z)- \big(\mathrm{Re}\,H_n(z)\big)_+ &\geq a_n \big(1-|\theta_n\big(\frac{z}{b_n}\big)\big|\big)
\\&
=a_n \int_\real  (1-\cos(uz))\hat{F}_n(du)
\\&
\geq  c(\delta)|z|^\kappa,
\end{split} \label{up1}
\end{equation}
 if \textbf{F}.a) holds true, or
\begin{equation}
\mathrm{Re}\,\psi(z)- \big(\mathrm{Re}\,H_n(z)\big)_+ =\mathrm{Re}\,\psi(z), \label{up2}
\end{equation}
if \textbf{F}.b) is satisfied. On the other hand, for $z\in (0,1)$ we have
\begin{align*}
|\ln z+1-z|&\leq \sum_{k=2}^\infty \frac{(1-z)^k}{k}\leq \frac{1-z}{2} \sum_{k=1}^\infty (1-z)^k
 \\&
 \leq \frac{(1-z)^2}{2z}.
 \end{align*}
Then by \eqref{thetan}  and \textbf{G}
we derive
\begin{equation}
\begin{split}
|\mathrm{Re}H_n(z)|&\leq  \Big|\mathrm{Re}\,\psi(z)-a_n\big(1-
|\theta_n(z/b_n)\big|\big)\Big|+ a_n \Big|\big(1-
|\theta_n(z/b_n)\big|\big)- \ln |\theta_n (z/b_n)| \Big|
\\&
\leq \Big| \int_\real (1-\cos(zu))u^{-2} (\hat{M}(du)-\hat{M}_n(du))\Big|
\\&
\quad + 2^{-1} a_n \Big( \int_\real (1-\cos(zu))\hat{F}_n(du)\Big)^2 \cdot \Big( \int_\real \cos (zu)\hat{F}_n(du)\Big)^{-1}
\\&
\leq c_1 \left(\gamma_n'(1+z^2)+ (2a_n )^{-1}\big( \RR \psi_n(z)\big)^2\right)
\\&
\leq  c_1\left(\gamma_n' (1+z^2 )+(2 a_n )^{-1} \Big((1+z^2)\gamma_n +  (1+z^2)  \Big)^2\right)
\\&
\leq c_2 (1+z^2)^2 (\gamma_n +a_n^{-1}). \label{reh}
\end{split}
\end{equation}

\noindent Next we estimate $|\mathrm{Im}\, H_n(z)|$. Observe that  for $z=x+iy$, where $x,y\in \real$,
$$
\mathrm{Im} \ln z= Arg \,z = \arctan\frac{y}{x},
$$
and  for all $x\in \real$ we have $|\arctan x -x |\leq c_3 |x|^3$,
where $c_3>0$ is some constant.
Therefore,
\begin{equation}
\begin{split}
|\mathrm{Im}&\, H_n(z)| = \Big| \mathrm{Im}\, \psi(z)+ a_n \mathrm{Im}\, \ln \theta_n\big(z/b_n\big)\Big|
\\&
\leq \Big| -\beta z + \int_\real \frac{z\sin u -\sin(zu)}{u^2} M(du) + a_n \arctan\frac{\mathrm{Im}\theta_n (z/b_n)}{\mathrm{Re} \theta_n (z/b_n)} \Big|
\\&
\leq \Big| -\beta z + z\int_\real \frac{\sin u}{u^2}M_n(du)\Big| +
\Big|\int_\real\frac{z\sin u - \sin (zu)}{u^2} (M_n -M)(du)\Big|
\\&
+ \Big| \int_\real \frac{\sin (zu)-z\sin u}{u^2} M_n(du)\Big| \Big| \frac{1}{\mathrm{Re} \theta_n(z/b_n)} -1\Big|
\\&
+ \Big| \int_\real z\frac{\sin u}{u^2} M_n(du)\Big| \Big| \frac{1}{\mathrm{Re} \theta_n(z/b_n)} -1\Big|
+ a_n \Big|\arctan\frac{\mathrm{Im}\theta_n (z/b_n)}{\mathrm{Re} \theta_n (z/b_n)} -\frac{\mathrm{Im}\theta_n (z/b_n)}{\mathrm{Re} \theta_n (z/b_n)}\Big|
\\&
\leq  I_{11}+ I_{12}+I_{13}+ I_{14}+I_{15}.\label{I-10}
\end{split}
\end{equation}
Observe that  by  $M_n\rightarrow M$ and $a_n\beta_n \to \beta$
(cf. \eqref{conv-mes})  we have
\begin{equation}
I_{11}(n)\leq |z||\beta-a_n \beta_n |=|z|\chi_n. \label{I11}
\end{equation}
For $I_{12}(n)$ we have
\begin{equation}
I_{12}(n) \leq \big| \II \phi_n(z)-\II \phi(z)| \leq c_4(1+ z^2) \gamma_n''\label{I21-1}
\end{equation}
Using  \textbf{G}, we derive
\begin{equation}
I_{13}(n)\leq c_5 a_n^{-1} |\II \phi_n(z)| \RR \phi_n(z)\leq c_6 a_n^{-1} (1+z^2)^2. \label{I31-1}
\end{equation}
Analogously,
\begin{equation}
I_{14}(n) \leq c_7 \beta_n  a_n^{-1}|z| \RR \phi_n(z) \leq c_8 a_n^{-1} |z|(1+z^2). \label{I41-1}
\end{equation}
Finally, for $I_{15}$ we have
\begin{align*}
I_{15}&\leq a_n c_3\Big|
\frac{\mathrm{Im}\theta_n (z/b_n)}{\mathrm{Re} \theta_n (z/b_n)}\Big|^3
\leq c_9 a_n  \Big|\int_\real \sin (zu)F_n(du)\Big|^3
\\&
\leq c_9 a_n^{-2} \Big|\int_\real \frac{\sin (zu)-z\sin u}{u^2} M_n(du)+ z\int_\real \frac{\sin u}{u^2}M_n(du)\Big|^3
\\&
= c_9 a_n^{-2} \Big| \mathrm{Im} \phi_n(z)-z\beta_n \Big|^3
\\&
\leq c_{10} a_n^{-2}(1+z^2)^3.
\end{align*}
  Thus, we arrive at
\begin{equation}
\begin{split}
I_1(n) &\leq c_{11} \max(\kappa_n,\gamma_n',\gamma_n'',a_n^{-1}) \int_0^{\delta b_n} e^{-c(\delta) z^\kappa}  (1+z^2)^3 dz
\\&
\leq c_{12} \max(\kappa_n,\gamma_n',\gamma_n'',a_n^{-1}). \label{I1}
\end{split}
\end{equation}

{\itshape Estimation of $I_2$}. We have  by \textbf{C} and \textbf{D}
\begin{align*}
I_2(n)&= \int_{|z|\geq \delta b_n} \Big|\theta_n\Big(\frac{z}{b_n}\Big)\Big|^{a_n} dz
=b_n \int_{|x|\geq \delta } |\theta_n(x)|^{a_n}dx
\\&
\leq b_n N(\delta)^{a_n-2} \int_{|x|\geq \delta} |\theta_n(x)|^2dx
\\&
\leq b_n N(\delta)^{a_n-2} \sup_{n\geq 1} \int_\real g_n^2(x)dx
\\&
=c_{13} b_ne^{a_n\ln N(\delta)}.
\end{align*}
Take $\varepsilon>0$ such that $\ln N(\delta)+\varepsilon<0$. By \textbf{E}, $\frac{\ln b_n}{a_n} \to 0$ as $n\to \infty$, and thus without loss of generality we may assume that  for all $n\geq 1$ we have $\frac{\ln b_n}{a_n}\leq \varepsilon$. Then
$$
b_n e^{a_n \ln N(\delta)} \leq e^{-a_n |\ln N(\delta)+\varepsilon|}.
$$

{\itshape Estimation of $I_3$}. For any $\epsilon>0$
\begin{align*}
I_3(n) &\leq c_{14}\int_{|z|\geq \delta b_n} e^{-\mathrm{Re}\psi(z)}dz
\leq  c_{15} (\epsilon) e^{-(1-\epsilon) \mathrm{Re}\psi(\delta b_n)}.
\end{align*}
Summarizing the estimates for $I_i(n)$, $i=1,2,3$, we derive
$\Delta_n\leq C  \rho_{\epsilon, \delta}(n)$.
\end{proof}

\medskip

\begin{proof}[Proof of Corollaries~\ref{cor2} and \ref{cor3}]
Clearly, the proofs are obtained as slight modifications  of the proof of Theorem~\ref{lt-main}. Since $a_n$ and $b_n$ satisfy \textbf{H}, the terms $I_2(n)$ and $I_3(n)$ decay as $o(n^{-k})$, $n\to\infty$, for any $k\geq 0$. This implies the statement of Corollary~\ref{cor2}. To complete the proof of Corollary~\ref{cor3}, we need to  estimate more precisely  $I_1(n)$.  Let us look closely on the properties of the function $H_n(z)$ from \eqref{Hnz}. Since
both $p_n(x)$ and $p(x)$ are symmetric, $H_n(z)$ is real-valued. Further, condition
\eqref{P} implies that $H_n(z)\geq 0$. Therefore, instead of \eqref{reh} we get
$$
H_n(z) \leq \psi(z)+a_n \ln \theta_n(z/b_n)\leq  \gamma_n'(1+z^2),
$$
which implies
$$
I_1(n)\leq C \gamma_n'.
$$
\end{proof}

\section{Examples}

\begin{example}\label{ex1}
Let $(\xi_n)_{n\geq 1}$ be i.i.d. random variables with probability density
$$
g(u)= c_\alpha \frac{(1-\cos u)}{ |u|^{1+\alpha}}, \quad u\in \real,\quad 0<\alpha<2.
$$
Then one can check  (using Theorem~\ref{fel1} with $a_n=n$ and  $b_n=n^{1/\alpha}$)  that
$$
S_n:= \frac{\xi_1+..+\xi_n}{n^{1/\alpha}}\Rightarrow S,
$$
where $S$ is a symmetric $\alpha$-stable distribution.  In this case the respective measure $M(du)$ in \eqref{LH1} is equal to $c_\alpha |u|^{1-\alpha} du$, and after the appropriate choice of $c_\alpha$ we have $\psi(z)=|z|^\alpha$. For example, in the case $\alpha=1$ we must chose $c_\alpha=1/\pi$. Clearly, conditions \textbf{A},  \textbf{B}, \textbf{D} and \textbf{G} are satisfied. Condition \textbf{C} is the Cramer condition (cf. \cite{Lu79}) for the characteristic function of $\xi_1$, which is satisfied since the law of $\xi_1$ is absolutely continuous.

Let us check condition \textbf{F}.   Consider
$$
\int_0^\infty  \frac{1-\cos (zu)}{u^{1+\alpha}} \cos (n^{1/\alpha}u)du  = |z|^\alpha \int_0^\infty \frac{1-\cos u}{u^{1+\alpha}} \cos (n^{1/\alpha}u/z)du.
$$
We need to estimate from above
$$
I(\alpha,k):=\int_0^\infty  \frac{1-\cos v}{v^{1+\alpha}} \cos (kv)dv.
$$
Note that for $\alpha=1$ we have  (cf. \cite{BE69}, p.28)
\begin{equation}
I(1,k)= \frac{\pi}{2} (1-|k|)_+.
\end{equation}
It is also possible to calculate $I(\alpha,k)$ for  $\alpha\in (0,2)\backslash \{1\}$.  Integrating by parts, we get for any $k>0$
\begin{align*}
I(\alpha,k)&=
 \frac{\sin kv}{k} \cdot \frac{1-\cos v}{v^{1+\alpha}}  \Big|_{0}^\infty- \frac{1}{k} \int_0^\infty \sin(kv) \frac{v\sin v - (1+\alpha)(1-\cos v)}{v^{2+\alpha}}dv
\\&
= -\frac{1}{k} \int_0^\infty \sin(kv) \frac{v\sin v - (1+\alpha)(1-\cos v)}{v^{2+\alpha}}dv.
\end{align*}
The integrals
$$
I_1(\alpha,k):= \int_0^\infty \frac{\sin (kv)\sin v}{v^{1+\alpha}}dv
$$
and
$$
I_2(\alpha,k):= \int_0^\infty \frac{\sin (kv)\sin^2(v/2)}{v^{2+\alpha}}dv.
$$
can be calculated explicitly,  see \cite{BE69}, p.77--78, from where one can derive the asymptotic behaviour as $k\to \infty$:
\begin{equation}
I_1(\alpha,k)= \frac{\pi}{4} \frac{|k+1|^\alpha-|k-1|^\alpha}{\Gamma(1+\alpha) \sin(\pi \alpha/2)}
\sim  c_{\alpha,1} |k|^{\alpha-1},
\end{equation}
\begin{equation}
I_2(\alpha,k)=
2^{-2} \Gamma(-1-\alpha) \cos(\pi \alpha/2) \big[ 2|k|^{\alpha+1}- |k+1|^{\alpha+1}- |k-1|^{\alpha+1}\big] \sim c_{\alpha,2} |k|^{\alpha-1},
\end{equation}
where $c_{\alpha,1}:= \frac{\pi}{2 \Gamma(1+\alpha) \sin(\pi \alpha/2)}$, $c_{\alpha,2}= 2^{-1}\alpha (\alpha+1)\Gamma(-1-\alpha) \cos(\pi \alpha/2)$.  Thus, we have the exact expression for $I(\alpha,k)$, from which we derive
$$
I(\alpha,k)\sim c_{3,\alpha} k^{\alpha-1} , \quad k\to \infty.
$$
where $c_{3,\alpha}= c_{1,\alpha} +2(\alpha+1)c_{2,\alpha} $. Finally,  for $|z|\leq \delta n^{1/\alpha}$ with $\delta>0$ is small enough,
\begin{align*}
a_n  \int_\real (1-\cos(zu)) \hat{F}_n(du) =  |z|^\alpha\left(1- 2 c_\alpha  I\big(\alpha, \frac{n^{1/\alpha}}{|z|}\big)\right)
\geq c(\delta) |z|^\alpha,\quad \alpha\in (0,2),
\end{align*}
where $c(\delta)>0$ is some constant.

Let us calculate the order of convergence. From above, we have for $\alpha\in (0,2)$
\begin{align*}
\left|\RR \psi(z)-\RR \psi_n(z)\right|  &= \Big|\int_\real \frac{(1-\cos (zu))}{|u|^{1+\alpha}} \cos (n^{1/\alpha}u)du\Big|
\leq
\frac{Cz^2 }{n^{(2-\alpha)/\alpha}}.
\end{align*}
Thus, by Corollary~\ref{cor2} we arrive at
\begin{equation}
\rho(n)\leq
\begin{cases}
C n ^{-1} , & 0<\alpha<1,
\\
C n^{-\frac{2-\alpha}{\alpha}},& 1\leq \alpha <2.
\end{cases}
\end{equation}

\end{example}

\begin{example}\label{ex2}
Suppose now that $\xi_{1,n}$ possesses the distribution density 
$$
g_n(u):=\tfrac{1}{2nu\sinh(u/n)}\I_{|u|\geq 1},
$$
 and 
$a_n=b_n=n$. Conditions \textbf{A}, \textbf{C}--\textbf{E} were already checked in \cite{KKu07},  in particular, it was shown that $S_n$ converges in distribution to a hyperbolic cosine distribution $S$, i.e. the distribution  density of $S$  is $p(x)= \tfrac{1}{\pi \cosh x}$.   Since in this case
$$
a_n f_n(u)= \frac{1}{2u \sinh u} 1_{|u|\geq \frac{1}{n}} \uparrow  \frac{1}{2u \sinh u}=:f(u) \quad \text{as $n\to\infty$},
$$
the function
$$
\psi(z)= \int_\real (1-\cos (uz)) f(u)du
$$
satisfies  condition \textbf{B} with $\alpha=1$.
 Let us check \textbf{F}
for $\kappa=1$. Since for $|z|\leq 1$ we have $1-\cos z \geq (1-\cos 1) z^2$, then estimating $\tfrac{u}{\sinh u}$ from below for small $u$ by a constant we get
\begin{align*}
a_n \int_\real (1-\cos (zu)) f_n(u)du &\geq n(1-\cos 1)  \int_{ |uz|\leq 1} (zu)^2 f_n(u)du
\\&
\geq (1-\cos 1) |z|\inf_{|z|\leq \delta  n } |z| \int_{1/n}^{1/|z|}  \frac{u}{\sinh u} du
\\&
\geq
c_1 |z|\inf_{|z|\leq \delta  n } |z|   \left( \frac{1}{|z|}-\frac{1}{n} \right)
\\&
\geq c_1 (1-\delta)|z|,
\end{align*}
uniformly in $\{z:\,\,|z|\leq \delta n\}$. Thus, condition \textbf{F} holds true.

 It remains to check condition \textbf{G}. Let $|z|\leq \delta$. Since the function $r\sinh (u/r)$ is increasing in $r$, we have by dominated convergence theorem
\begin{align*}
\inf_{n\geq 1}|\theta_n(z)| &= \inf_{n\geq 1} \Big|\int_{u\geq 1} \frac{\cos (zu)}{nu\sinh (u/n)}du\Big|=\lim_{n\to \infty} |\theta_n(z)| = \int_1^\infty \frac{\cos(zu)}{u^2}du
\\&
\geq \cos 1 \int_1^{1/\delta} \frac{du}{u^2},
\end{align*}
 which gives \textbf{G}.

Finally, by Corollary~\ref{cor2} we arrive at
$$
\sup_{x\in \real } |p_n(x)-p(x)|\leq \frac{C}{n}.
$$
\end{example}

\textbf{Acknowledgement.} The author thanks A.Kulik, who initiated this work, for helpful discussions and comments.

\end{document}